\documentclass[a4paper,12pt,reqno]{article}
\usepackage[T1]{fontenc}
\usepackage{authblk}
\usepackage{epsfig,graphicx,subfigure}
\usepackage{amsthm,amsmath,amssymb,amsfonts}
\usepackage{color,xcolor}
\usepackage[all]{xy}
\usepackage[top=40mm, bottom=40mm, left=40mm, right=35mm]{geometry}
\usepackage{hyperref}

\newtheorem{theorem}{Theorem}[section]

\newtheorem{corollary}[theorem]{Corollary}
\theoremstyle{definition}

\numberwithin{equation}{section}

\begin{document}
\setcounter{page}{1}
\title{\vspace{-1.5cm}
\vspace{.5cm}
\vspace{.7cm}
{\large{\bf On $\theta$-centralizing $\theta$-generalized derivations on convolution algebras} }}
\date{}
\author{{\small \vspace{-2mm} M. Eisaei$^1$, M. J. Mehdipour$^2$\footnote{Corresponding author} and Gh. R. Moghimi $^3$}}
\affil{\small{\vspace{-4mm} $^{1, 3}$ Department of Mathematics, Payame Noor University (PNU),
Tehran, Iran }}
\affil{\small{\vspace{-4mm} $^1$ Department of Mathematics, Shiraz University of Technology, 71555-313, Shiraz, Iran. }}
\affil{\small{\vspace{-4mm} mojdehessaei59@student.pnu.ac.ir}}
\affil{\small{\vspace{-4mm} mehdipour@sutech.ac.ir}}
\affil{\small{\vspace{-4mm} moghimimath@pnu.ac.ir}}
\maketitle
\hrule
\begin{abstract}
\noindent
Let $\theta$ be an isomorphism on $L_0^{\infty} (w)^*$. 
In this paper, we investigate $\theta$-generalized derivations on 
$L_0^{\infty} (w)^*$. We show that every $\theta$-centralizing $\theta$-generalized derivation on $L_0^{\infty} (w)^*$ is a $\theta$-right centralizer. We also prove that this result is true for $\theta$-skew centralizing $\theta$-generalized derivations. 
\end{abstract}

\\
\hrule
\vspace{0.5 cm}
\baselineskip=.8cm
\section{Introduction}
Let $w:[0, \infty) \rightarrow [1, \infty)$ be a continuous function such that $w(0)=1$ and for every $x,y \geq 0$ 
$$w(x+y) \leq w(x) w(y).$$
Let $L^1(w)$ be the Banach algebra of all Lebesgue measurable functions on $[0, \infty)$. 
Let also $M(w)$ be the Banach algebra of all complex regular Borel measure on $[0, \infty)$; for study of these Banach algebras see[5, 13]. 
We denote by $L_0^{\infty} (w)$ the Banach space of all Lebesgue measurable functions $f$ on $[0, \infty)$ such that 
\begin{equation*} 
\hbox{ess sup} \{ f(y) \chi_{(x, \infty)}(y) / w(y): ~ y \geq 0 \big \} \rightarrow 0
\end{equation*}
as $x \rightarrow + \infty$, where 
$\chi_{(x, \infty)}$ is the characteristic function on $(x, \infty)$. 
It is proved that the dual of $L_0^{\infty} (w)$, represented by $L_0^{\infty} (w)^*$, is a Banach algebra with the first Arens product defined by 
$$mn(f)=m(nf),$$
where the functional $nf$ is defined by
$nf(\varphi)= n(f \varphi)$, in which 
$$f \varphi(x)= \int_{0}^{\infty} f(x+y) \varphi(y) dy$$
for all $m,n \in L_0^{\infty} (w)^*$, $f \in L_0^{\infty} (w)$, $\varphi \in L^1(w)$ and $x \geq 0$; for more details see [8, 9]. 
By the usual way, $L^1(w)$ may be regarded as a subspace of 
$L_0^{\infty} (w)^*$. 
In this case, $L^1(w)$ is a closed ideal of $L_0^{\infty} (w)^*$. 
Note that the sequence $\{i \chi_{(0, 1/i)} \}_{i \in \Bbb{N}}$ is a bounded approximate identity for $L^1(w)$. 
The set of all weak$^*$-cluster points of an approximate identity of $L^1(w)$ 
bounded by one is denoted by $\Lambda(L_0^{\infty} (w)^*)$. 
It is easy to see that $u \in \Lambda(L_0^{\infty} (w)^*)$
if and only if $u$ is a right identity for $L_0^{\infty} (w)^*$. 

Let $\theta$ be a homomorphism on $L_0^{\infty} (w)^*$ and $T$ be a linear map on $L_0^{\infty} (w)^*$.
Then $T$ is called a \textit{ $\theta$-right centralizer} if for every $m , n \in L_0^{\infty} (w)^*$
\begin{eqnarray*}
T(mn)= \theta(m) T(n).
\end{eqnarray*}
A linear map $d$ on $L_0^{\infty} (w)^*$ is called a
\textit{$\theta$-derivation} if
\begin{eqnarray*}
d(mn)=d(m) \theta(n) + \theta(m) d(n)
\end{eqnarray*}
for all $m,n \in L_0^{\infty} (w)^*$.
Also, a linear map $D$ on $L_0^{\infty} (w)^*$ is called a
\textit{$\theta$-generalized derivation with associated to derivation $d$}, if
\begin{eqnarray*}
D(mn)=\theta(m) D(n)+ d(m) \theta(n)
\end{eqnarray*}
for all $m,n \in L_0^{\infty} (w)^*$.
We denote this concept with $(D, d)$.
Note that if $D=d$, then $D$ is a $\theta$-derivation. 
In the case where $d=0$, then for every $m,n \in L_0^{\infty} (w)^*$ 
$$D(mn)=\theta(m) D(n).$$
This type of $\theta$-generalized derivation is called a \textit{$\theta$-right centralizer}. 
A linear map $T$ on $ L_0^{\infty} (w)^*$ is called \textit{$\theta$-commuting} if for every $m \in L_0^{\infty} (w)^*$
\begin{eqnarray*}
T(m) \theta(m)=\theta(m) T(m),
\end{eqnarray*}
and $T$ is called \textit{$\theta$-centralizing} if
\begin{eqnarray*}
[T(m), \theta(m) ] := T(m) \theta(m)- \theta(m) T(m) \in Z (L_0^{\infty} (w)^*)
\end{eqnarray*}
for all $m \in L_0^{\infty} (w)^*$, where $Z (L_0^{\infty} (w)^*)$ is the center of $L_0^{\infty} (w)^*$.

Some authors studied the Banach algebra $L_0^{\infty} (w)^*$ [1,2,10-12]. 
For example, Ahmadi Gandomani and the second author [1] studied generalized derivations on $L_0^{\infty} (w)^*$. 
They showed that every centralizing generalized derivation on 
$L_0^{\infty} (w)^*$ is a right centralizer. They proved that this result holds for skew centralizing generalized derivations; see also [3,6,7]. 
In this paper, we investigate these facts for $\theta$-generalized derivations on $L_0^{\infty} (w)^*$.


\section{Main Results}

In the following, let $\text{ran} (L_0^{\infty} (w)^*)$ be the right annihilator of $L_0^{\infty} (w)^*$, that is, the set of all 
$r \in L_0^{\infty} (w)^*$ such that $mr=0$ for all $m \in L_0^{\infty} (w)^*$.
\begin{theorem} \label{m1}
Let $\theta $ be a homomorphism on $L_0^{\infty} (w)^*$ and $(D, d)$ be a $\theta$-generalized derivation on
$L_0^{\infty} (w)^*$. Then the following statements hold.

\emph{(i)}
$D$ maps $\emph{ran} (L_0^{\infty} (w)^*)$ into $\emph{ran} (L_0^{\infty} (w)^*)$.

\emph{(ii)}
$D$ is $\theta$-centralizing if and only if $D$ is $\theta$-commuting.

\emph{(iii)}
If $\theta$ is an isomorphism and $D$ is $\theta$-centralizing, then $D$ is a $\theta$-right centralizer.
\end{theorem}
\begin{proof}
\text{(i)} Let $k \in L_0^{\infty} (w)^*$,
$u \in \Lambda(L_0^{\infty} (w)^*)$ and
$r \in \text{ran}(L_0^{\infty} (w)^*)$. Then 
\begin{equation*}
k D(r) = k \theta(u) D(r)= k [D(ur)- d(u) \theta(r)]=0,
\end{equation*}
where $u \in \Lambda(L_0^{\infty} (w)^*)$.
Hence $D(r) \in \text{ran}(L_0^{\infty} (w)^*)$.

$\text{(ii)}$ Let $m \in L_0^{\infty} (w)^*$ and
\begin{equation*}
[D(m), \theta(m)] \in Z(L_0^{\infty} (w)^*).
\end{equation*}
So if $u \in \Lambda(L_0^{\infty} (w)^*)$, then
\begin{eqnarray*}
[D(m), \theta(m)] &=& [ D(m), \theta (m)] u\\
& =& u [ D(m), \theta (m)] \\
&= & u( D(m) \theta(m)- \theta(m) D(m)) \\
&=& u D(m) \theta(m) - u D(m) \theta(m) \\
&=&0.
\end{eqnarray*}
It follows that $D$ is $\theta$-commuting.

$\text{(iii)}$ Let $\theta$ be an isomorphism and $D$ be $\theta$-centralizing.
Then for every $m \in L_0^{\infty} (w)^*$
\begin{equation*}
D(m) \theta(m)= \theta(m) D(m).
\end{equation*}
Choose $u \in \Lambda(L_0^{\infty} (w)^*)$. Then
\begin{equation} \label{m1}
D(u)= D(u) \theta(u)= \theta(u) D(u).
\end{equation}
On the other hand,
\begin{equation} \label{m2}
D(u)= D(uu)= \theta(u) D(u) + d(u) \theta(u)= \theta(u) D(u)+ d(u).
\end{equation}
From this and (2.1), we have $d(u)=0$.
Note that if $r \in \text{ran}(L_0^{\infty} (w)^*)$, then
\begin{eqnarray*}
(r+u)^2 =r+u
\end{eqnarray*}
and $\theta(u+r)$ is a right identity for $L_0^{\infty} (w)^*$.
So
\begin{eqnarray*}
D(u+r) &=& D(u+r) \theta(u+r) \\
&=& D((u+r)^2)- d(u+r) \theta(u+r) \\
&=& D(u+r) - d(r).
\end{eqnarray*}
Hence $d(r)=0$.
This shows that for every $m \in L_0^{\infty} (w)^*$,
\begin{eqnarray*}
d(m) &=& d(um) \\
&=& d(u) \theta(m) + \theta(u) d(m) \\
&=& \theta(u) d(m) \\
&=&0,
\end{eqnarray*}
because $d$ maps $L_0^{\infty} (w)^*$ into $\text{ran}(L_0^{\infty} (w)^*)$; see [12]. Therefore,
\begin{equation*}
D(m)= D(mu) = \theta(m) D(u) + d(m) \theta (u) = \theta(m) D(u)
\end{equation*}
for all $m \in L_0^{\infty} (w)^*$.
That is, $D$ is a $\theta$-right centralizer.
\end{proof}
Let $\theta$ be a homomorphism on $L_0^{\infty} (w)^*$. Then 
a map $T: L_0^{\infty} (w)^* \rightarrow L_0^{\infty} (w)^*$ is called \textit{$\theta$-skew commuting} if for every 
$m \in L_0^{\infty} (w)^*$ 
\begin{equation*} 
\langle T(m), \theta(m) \rangle := T(m) \theta(m) + \theta(m) T(m)=0,
\end{equation*}
if for every $m \in L_0^{\infty} (w)^*$, 
\begin{equation*} 
\langle T(m), \theta(m) \rangle \in Z(L_0^{\infty} (w)^*), 
\end{equation*}
then $T$ is called \textit{$\theta$-skew cenralizing}. 
\begin{theorem} \label{m2}
Let $\theta$ be an isomorphism on $L_0^{\infty} (w)^*$ and
$(D, d)$ be a $\theta$-generalized derivation on $L_0^{\infty} (w)^*$. Then the following statements hold.

\emph{(i)}
If $D$ is $\theta$-skew centralizing, then there exists $n \in Z (L_0^{\infty} (w)^*)$ such that
$D(m)=mn$ for all $m \in L_0^{\infty} (w)^*$.

\emph{(ii)} If $D$ is $\theta$-skew commuting, then $D=0$ on
$L_0^{\infty} (w)^*$.
\end{theorem}
\begin{proof}
$\text{(i)}$. Let $m \in L_0^{\infty} (w)^*$ and
\begin{equation*}
\langle D(m), \theta(m) \rangle \in Z (L_0^{\infty} (w)^*).
\end{equation*}
Then
\begin{equation*}
0= [ \langle D(m), \theta(m) \rangle, \theta(m) ]= [ D(m), \theta(m)^2].
\end{equation*}
It follows that
\begin{equation*}
D(u)= D(u) \theta(u)^2= \theta(u)^2 D(u)= \theta(u) D(u).
\end{equation*}
On the other hand,
\begin{equation*}
D(u)= D(uu) = \theta(u) D(u)+ d(u).
\end{equation*}
Hence $d(u)=0$.
An argument similar to the proof of Theorem 2.1, shows that $D(m)= \theta(m) D(u)$ for all $m \in L_0^{\infty} (w)^*$.
But,
\begin{eqnarray*}
\langle D(u), \theta(u) \rangle &=& D(u) \theta(u) + \theta(u) D(u) \\
&=& \theta(u) D(u)\theta(u) + \theta(u) D(u) \\
&=& 2 \theta(u) D(u)
\end{eqnarray*}
is an elemnt of $Z(L_0^{\infty}(w)^*)$. Since
\begin{equation*}
D(m)= \theta(m) D(u) = \theta(m) \theta(u) D(u),
\end{equation*}
the statement $\hbox{(i)}$ holds.

$\hbox{(ii)}$ Let $D$ be $\theta$-skew commuting.
Then there exists $n \in Z(L_0^{\infty}(w)^*)$ such that
\begin{equation*}
D(m)=mn
\end{equation*}
for all $m \in L_0^{\infty}(w)^*$.
So, if $u \in \Lambda(L_0^{\infty}(w)^*)$, then
\begin{eqnarray*}
0&=& D(u) \theta(u) + \theta(u) D(u) \\
&=& un \theta (u) + \theta(u) u n\\
&=& u \theta (n) n + \theta (u) n \\
&=& \theta(u) n + \theta (u) n \\
&=& 2 n \theta (u) \\
&=& 2n.
\end{eqnarray*}
Hence $n=0$ and therefore, $D=0$.
\end{proof}
\begin{theorem} \label{m4}
Let $\theta$ be a homomorphism on $L_0^{\infty}(w)^*$ and
$(D, d)$ be a $\theta$-generalized derivation on
$L_0^{\infty}(w)^*$. Then $D$ is a $\theta$-derivation if and only if $D=d$.
\end{theorem}
\begin{proof}
Let $D$ be a $\theta$-derivation.
Then for every $m \in L_0^{\infty}(w)^*$, we have
\begin{eqnarray*}
D(m)= D(m.u) &=& D(m) \theta (u)+ \theta(m) D(u) \\
&=& D(m)+ \theta (m) D(u)
\end{eqnarray*}
and
\begin{equation*}
D(m)= \theta(m) D(u)+d(m) \theta(u) = \theta(m) D(u)+ d(m).
\end{equation*}
Hence $d=D$.
\end{proof}
We denote by $\hbox{rad}(L_0^{\infty}(w)^*)$ the radical of 
$L_0^{\infty}(w)^*$. Before, we give the next result, let us recall that a map $T: L_0^{\infty}(w)^* \rightarrow L_0^{\infty}(w)^*$ is called \emph{spectrally infinitesimal} if $r(T(m))=0$ for all $m \in L_0^{\infty}(w)^*$, where $r(.)$ is the spectral radius. 

\begin{corollary} \label{m5}
Let $\theta$ be a homomorphism on $L_0^{\infty} (w)^*$ and $(D, d)$ be a $\theta$-generalized derivation on $L_0^{\infty} (w)^*$.
If $D$ maps $L_0^{\infty} (w)^*$ into $\emph{rad}(L_0^{\infty}(w)^*)$ or $D$ is spectrally infinitesimal, then $D$
is a $\theta$-derivation.
\end{corollary}
\begin{proof}
First, note that if $D$ maps $L_0^{\infty} (w)^*$ into
$$\text{rad}(L_0^{\infty}(w)^*)= \text{ran}(L_0^{\infty}(w)^*),$$ then $D$ is spectrally infinitesimal.
Hence $r(D(m))=0$ for all $m \in L_0^{\infty} (w)^*$.
One can prove that
$\frac{L_0^{\infty} (w)^*}{ \text{ran}(L_0^{\infty}(w)^*)}$
is isomorphic to the commutative Banach algebra $M(w)$; see [9].
In view of [14], there exists $c>0$ such that
\begin{equation*}
r(m D(u)) \leq c \, r(m) \, r(D(u))=0
\end{equation*}
for all $m \in L_0^{\infty}(w)^*$ and
$u \in \Lambda(L_0^{\infty}(w)^*)$.
It follows from Proposition 25.1 (ii) of [4] and 
that
$$D(u) \in \text{rad}(L_0^{\infty}(w)^*)= \text{ran}(L_0^{\infty}(w)^*).$$
Therefore,
\begin{equation*}
D(m)= \theta(m) D(u)+ d(m)= d(m)
\end{equation*}
for all $m \in L_0^{\infty}(w)^*$.
\end{proof}

\end{document}